\documentclass[12pt]{article}
\usepackage{layout,yfonts, graphicx, amsmath, amsthm, amssymb, euscript, enumerate, bbold, bbm, mathrsfs}
\usepackage{mathtools}
\usepackage{enumitem}
\usepackage{esvect}
\usepackage[title]{appendix}
\numberwithin{equation}{section}
\usepackage{xcolor}
\usepackage{appendix}
\usepackage{amssymb}
\definecolor{webgreen}{rgb}{0,.5,0}
\definecolor{webbrown}{rgb}{.8,0,0}
\definecolor{emphcolor}{rgb}{0.95,0.95,0.95}

\usepackage{blindtext} 
\usepackage{scrextend}

\addtokomafont{labelinglabel}{\sffamily}

\usepackage[colorinlistoftodos, textwidth=21mm, shadow,textsize=footnotesize]{todonotes}
\reversemarginpar
\usepackage[margin=1in]{geometry}
\theoremstyle{plain}
  \newtheorem{teor}{Theorem}[section]
  
  \newtheorem{prop}[teor]{Proposition}
  \newtheorem{lema}[teor]{Lemma}
  \newtheorem{coro}[teor]{Corollary}
\theoremstyle{remark}  
  \newtheorem{rem}[teor]{Remark}
\theoremstyle{plain}

\newtheoremstyle{hyp}{}{}{\itshape}{}{}{}{0pt}{}
\theoremstyle{hyp}

\DeclareMathAlphabet{\mathpzc}{OT1}{pzc}{m}{it}

\newcommand{\BE}{\begin{equation}}

\newcommand{\EE}{\end{equation}}
\newcommand {\BA}{\begin{align}}
\newcommand{\EA}{\end{align}}

\title{{Optimality of threshold strategies for spectrally negative L\'evy processes and a positive terminal value \\ at creeping ruin}}

\author{ Chongrui Zhu \footnote{Corresponding author: chongruizhu5@gmail.com}\\ }

\date{2023.01.04}
\begin{document}	
\maketitle
\begin{abstract}
This paper investigates a dividend optimization problem with a positive creeping-associated terminal value at ruin for spectrally negative L\'evy processes. We consider an insurance company whose surplus process evolves according to a spectrally negative L\'evy process with a Gaussian part and a finite L\'evy measure. Its objective function relates to dividend payments until ruin and a creeping-associated terminal value at ruin. The positive creeping-associated terminal value represents the salvage value or the creeping reward when creeping happens. Owing to formulas from fluctuation theory, the objective considered is represented explicitly. Under certain restrictions on the terminal value and the surplus process, we show that the threshold strategy should be the optimal one over an admissible class with bounded dividend rates. 
\end{abstract}
\section{Introduction}
The classical optimal dividends problem surveys the candidate strategy that maximizes the objective concerning cumulative dividend payments over an admissible class with or without the bound on dividend rates. In recent years, it has become an area that received renewed attention due to the applicability of fluctuation theory to characterize the expected cumulative discounted dividend payments and discounted terminal value when the risk surplus process follows a spectrally one-sided L\'evy process. Note that spectrally one-sided Le\'vy processes include the spectrally negative and positive L\'evy processes. In addition to the reason that the spectrally one-sided L\'evy process is of greater analytical tractability, it becomes a common choice for the risk surplus process since it serves as a natural generalization of the classical compound Poisson process. Within collective risk theory, negative jumps of the latter model policyholders' claims. Many remarkable efforts have been made on this topic, and we do not attempt to give an overall description of various optimality results here but instead refer the interested readers to the representative literature, such as \cite{czarna2014dividend,avram2007optimal,bayraktar2013optimal,bayraktar2014optimal,egami2010solving,noba2018optimal,avanzi2016optimal,li2019optimal,loeffen2008optimality}.

A modified version of the classical dividend optimization problem is to consider an additional terminal value in the objective function. The terminal value term can be understood as the reward or the penalty at some given epoch, depending on whether it is positive or not. A dividend optimization problem with a terminal value at ruin for spectrally negative L\'evy processes was studied in \cite{loeffen2009optimal}, where the admissible dividend strategy is without the ceiling dividend rate. \cite{yin2013optimal} investigated the same topic for spectrally positive L\'evy processes. For the case where the admissible class is with bounded dividend rates, \cite{junca2019optimality} surveyed a constrained dividend optimization problem for spectrally one-sided L\'evy processes. \cite{loeffen2010finetti} researched the optimal dividends problem in the more general framework, considering an extra affine Gerber-Shiu function in the objective. This affine function describes the benefits or penalties relevant to ruin behavior by tracking the deficit level linearly, which motivates our research of creeping behavior, in which a terminal value is bestowed when the deficit level equals zero exactly.  

The terminal value in our problem setting is positive when creeping occurs. Such a set-up has a realistic background. A certain amount of financial compensation is provided at the moment when the random volatility of the risk reserve rather than abrupt claims from policyholders leads to the ultimate bankruptcy. In other words, only when the ruin occurs in a "continuous" and more predictable manner would the salvage value from the insurance company be transferred to the beneficiaries who receive dividend payments as well. This “continuous” transferring corresponds to the case that such storage value is easier to access when the ruin is less severe. The explanation of the positive terminal value as a salvage value can be found in \cite{loeffen2009optimal,RADNER19961373}. As seen, the positive terminal value here ignores the jumps' effect (brutal ruin) of the underlying risk surplus. Alternatively, this represents the preference of potential shareholders as the brutal ruin is even less favored than the creeping one due to the higher uncertainty engendered in the brutal ruin. Moreover, such bias toward creeping decays with the lifetime of the surplus as the latter increasingly compensates for risks of terminating dividend flows unpredictably caused by brutal ruin. Here, the creeping-dependent salvage value/ creeping award setting fits reality more in the sense that the terminal value under consideration allows for the company’s economic situation or the shareholder's belief. 

In this article, the risk surplus process of the insurance company $X=(X_t)_{t\geq 0}$ is assumed to evolve as a spectrally negative L\'evy process. The spectrally negative L\'evy process is the stochastic process that has stationary, independent increments and no positive jumps, and it can be identified with the L\'evy triplet $(\gamma,\sigma,\Pi)$, where $\gamma\in \mathbb{R}$, $\sigma\geq  0$ and the measure $\Pi$ concentrated on $(0,\infty)$ is such that 
\begin{gather*}
    \int_0^{\infty}\left(1\wedge x^2\right)\Pi(dx)<\infty.
\end{gather*}
Let $(\Omega,\mathcal{F},\mathbb{F}=(\mathcal{F}_t)_{t\geq 0},\mathbb{P})$ be the complete filtered probability space satisfying common assumptions, and by $\mathbb{P}_x$ and $\mathbb{E}_x$ denote the probability law and expectation operator conditioned on $\{X_0=x\}$, respectively. For $t,\theta \geq 0$, the Laplace exponent $\psi_X$ of $X$ is given by 
\begin{gather*}
\mathbb{E}_x\left[e^{\theta X_t}\right]
=e^{\theta x+\psi_X(\theta) t},
\end{gather*}
where 
\begin{gather*}
    \psi_X(\theta)=\gamma \theta +\frac{1}{2}\sigma^2\theta^2
    -\int_{(0,\infty)}
    \left(1-e^{-\theta z}-\theta z1_{\{0<z<1\}}\right)\Pi(dz).
\end{gather*}
Note that in the subsequent text the drift of $X$ will be defined as $c=\gamma+\int_{(0,1)}z\Pi(dz)$ if $\int_{(0,1)}z\Pi(dz)<\infty$.

The \emph{dividend strategy} $\pi=(L^{\pi}_t)_{t\geq 0}$ is a non-decreasing, left-continuous, and $\mathbb{F}$-adapted process with $L_{0}=0$, where $L^{\pi}_t$ represents cumulative dividend payments generated by the company up to time $t$. The controlled surplus process $U^{\pi}=(U^{\pi}_t)_{t\geq 0}$ with dividend strategy $\pi$ executed is defined by $U^{\pi}_t=X_t-L_{t}^{\pi}$ for $t\geq 0$. Let $\tau_{\pi}=\inf\{t>0:U_t^{\pi}<0\}$ be the ruin time and the objective with a positive creeping-associated terminal value at ruin for dividend strategy $\pi$ is formulated as 
\begin{gather*}
V_{\pi}^S(x)
=
    \mathbb{E}_x\left[
    \int_{[0,\tau_{\pi}]} 
    e^{-qt}dL_t^{\pi}
    +
    e^{-q\tau_{\pi}}\Lambda_S(U_{\tau_\pi}^{\pi})1_{\{\tau_{\pi}<\infty\}}\right],
    \tag{1.1}
    \label{1.1}
\end{gather*}
where $q>0$ is the discounting rate, $S>0$ is the terminal value, and the related function $\Lambda_S:(-\infty,0]\to \mathbb{R}$ in \eqref{1.1} is defined as
\begin{gather*}
    \Lambda_S(x)=S1_{\{x=0\}},\text{  for all  } x\leq 0.
\end{gather*}
It follows from definition that $V_{\pi}^S(x)=\Lambda_S(x)=0$ for $S>0$ and $x< 0$. In \cite{loeffen2010finetti}, the equivalent of $\Lambda_S$ is an affine function, while here, $\Lambda_S$ vanishes on the non-positive real axis except for the origin. The event $\{U_{\tau_\pi}^\pi =0\}$ is called creeping or creeping downwards. As seen in \eqref{1.1}, $S>0$ represents the salvage value transferred to the beneficiary or the creeping bonus based on the preference of shareholders when creeping happens. The admissible class for absolutely continuous dividend strategy/ dividend strategy with bounded dividend rate $\pi$ is defined by
\begin{gather*}
    \mathcal{D}:=\{\pi=(L_t^{\pi})_{t\geq 0}:
    L^{\pi}_t=\int_0^t l^{\pi}_sds,
    \text{ }l^{\pi}_s\in [0,\delta]
     \text{ for all }s\geq 0
    ,
    \text{  and  }
    l^{\pi}_s=0\text{ for all }
    s\geq \tau_{\pi}
    \},
\end{gather*}
where $\delta>0$ is the maximal admissible dividend rate. The aim is to characterize the optimal dividend strategy $\pi_* \in \mathcal{D} $ $s.t.$
\begin{gather*}
V_{\pi_*}^S(x)=V^S(x):=\underset{\pi \in \mathcal{D}}{\mathrm{sup}}~V_{\pi}^S(x),\quad x\geq 0.
\tag{1.2}
\label{1.2}
\end{gather*}
In other words, we aim to show the optimality of the candidate admissible dividend strategy that maximizes the expected sum of dividend payments until ruin and a creeping-associated terminal value at ruin. Despite the fact that the specific form of risk surplus processes varies, the candidate solution for this type of problem shall be the threshold strategy, which means that any surplus over the threshold level is paid with the maximal admissible dividend rate, while nothing is paid whenever the surplus is under the threshold level, see, for example, \cite{junca2019optimality,kyprianou2012optimal,zhu2014dividend,zhu2015dividend}. We denote the threshold dividend strategy by $\pi_b=(L_t^b)_{t\geq 0}$ here. More precisely, the risk surplus process $U^b=(U^b_t)_{t\geq 0}$ with the threshold strategy executed evolves as
\begin{gather*}
    U_t^b=
    X_t-L_t^b=
    X_t-
    \int_0^t l^{b}_sds,
    \quad 
    \text{and}
    \quad
    l^{b}_s=
    \delta1_{\{U_s^b>b\}}.
\end{gather*}
A comprehensive study on various expected net present values for $U_b$ is shown \cite{kyprianou2010refracted}, where $U_b$ is referred to as the refracted L\'evy process. Furthermore, from Theorem 7 in \cite{kyprianou2010refracted}, spectrally negative L\'evy processes without the Gaussian part creep downward with zero probability. In what follows, the survey is restricted in the case $\sigma>0$ to avoid triviality. For the function $\omega:\mathbb{R}\to\mathbb{R}$, which is twice continuously differentiable on $(0,\infty)$, the infinitesimal generator of $X$ acting on it is given by
\begin{gather*}
    \mathcal{G}\omega(x)
    =\gamma \omega^{\prime}(x)
    +\frac{\sigma^2}{2}\omega^{\prime\prime}(x)
    +\int_{(0,\infty)}
    \left[ 
    \omega(x-y)-\omega(x)+\omega^{\prime}(x)y1_{\{0<y<1\}}\right]\Pi(dy),
    \quad 
    x>0.
\end{gather*}

Interestingly, the research inevitably involves higher-order derivatives of $q$-scale functions, which is the principal analytical tool in this article, and relies on some results of those. To illustrate, since the desired representation of creeping-related quantities in \eqref{1.1} relates to the second-order derivative at the zero point of $q$-scale functions, the need is to provide the corresponding exact value. It can be shown heuristically, and the precise proof is in Remark \ref{rem2.5}. Notably, technical difficulties are required to be solved when verifying the optimality condition. More specifically, the concept given in Proposition \ref{lem2.end}, which says that the $q$-scale function's derivative is the eigen-function of the underlying process, is in need for deducing the proposed optimality condition. It is worth mentioning that such smoothness results demonstrated here might be of independent interest and use in various modeling problems. The smoothness result can lead to the third-order derivative at the origin for $q$-scale functions. By appealing to that value, one can obtain critical consequences regarding the selection of the threshold level, which can be seen in the proof of Lemma \ref{Lem3.4}.

The argument in this article follows a standard "guess-and-verify" procedure. As mentioned before, the candidate policy shall be the threshold dividend strategy and the threshold level is determined according to certain criteria. The assertion is to ensure that such criteria entail the optimality of the threshold strategy over the admissible class, which is to be checked by proving the linked value function is the solution to the Hamilton-Jacobi-Bellman (HJB) equation.

The outline of this article is structured as follows. To start with, we introduce some key identities on $q$-scale functions from fluctuation theory for spectrally negative L\'evy processes in Section \ref{S.2}. After that, in Section \ref{S.3}, we present the condition that can verify the optimality of the candidate threshold strategy and proof that it is indeed valid. A numerical example and the conclusion will also be given in Section \ref{S.4} and Section \ref{S.5}, respectively. The technical proof would be deferred to the appendix.
\section{Preliminaries on scale functions}
\label{S.2}
The main analytical tool employed in this paper, the $q$-scale function for spectrally negative L\'evy process $X=(X_t)_{t\geq 0}$, $W^{(q)}:\mathbb{R}\to [0,\infty)$ with $q\geq 0$, is the strictly increasing continuous function defined by the Laplace transform as follows:
\begin{gather*}
    \int_0^{\infty}e^{-\theta x}W^{(q)}(x)dx
    =\frac{1}{\psi_X(\theta)-q}, 
    \quad
    \theta>\Phi(q),
    \tag{2.1}
    \label{2.1}
\end{gather*}
where the right inverse of $\psi_X$ is given by $\Phi(q)=\sup\{\theta\geq 0:\psi_X(\theta)=q\}$ and $W^{(q)}(x)=0$ for $x<0$. Similarly, the $q$-scale function for the surplus process perturbed by the ceiling dividend rate $Y=(X_t-\delta t)_{t\geq 0}$ is denoted by $\mathbf{W}^{(q)}$, and the right inverse of the Laplace exponent for $Y$ is given by
\begin{gather*}
    \phi(q)=\sup\{ 
    \theta\geq 0:\psi_Y(\theta)=\psi_X(\theta)-\delta \theta =q
    \}.
    \tag{2.2}
    \label{2.2}
\end{gather*}
In particular, $\phi(q)>\Phi(q)>0$ holds based on the strict convexity of the function $\psi_X$. The associated functions ${\overline{W}^{(q)}}:\mathbb{R}\to[0,\infty)$ and $\overline{\mathbf{W}}^{(q)}:\mathbb{R}\to [0,\infty)$ are given by 
\begin{gather*}
\overline{W}^{(q)}(x)=\int_0^x{W}^{(q)}(y)dy,\quad
    \overline{\mathbf{W}}^{(q)}(x)=\int_0^x{\mathbf{W}}^{(q)}(y)dy.
\end{gather*}
Drawing upon the proof of Theorem 4 and 5 in \cite{kyprianou2010refracted}, we have, for $x\in \mathbb{R}$,
\begin{gather*}
    \delta \int_0^x \mathbf{W}^{(q)}(x-y)W^{(q)}(y)dy
    =\overline{\mathbf{W}}^{(q)}(x)
    -\overline{W}^{(q)}(x),
    \tag{2.3}
    \label{2.3}
\end{gather*}
which is an immediate result of the observation that the Laplace transform for both sides of \eqref{2.3} is equal. The limit relation as follows holds, 
\begin{gather*}
\lim_{x\to \infty}
    e^{-\Phi(q)x}W^{(q)}(x)= \frac{1}{\psi_X^{\prime}(\Phi(q))},
    \quad
\lim_{x\to\infty}
     e^{-\phi(q)x}\mathbf{W}^{(q)}(x)=\frac{1}{\psi^{\prime}_Y(\phi(q))}= \frac{1}{\psi_X^{\prime}(\phi(q))-\delta},
     \tag{2.4}
     \label{2.4}
\end{gather*}
which can be found in Equation 3.7 of \cite{junca2019optimality}. Next, we will review some conclusions in \cite{loeffen2009optimal}, \cite{junca2019optimality}, and \cite{kyprianou2012optimal}.
\begin{lema}
\label{lem2.0}
Assume that the L\'evy measure $\Pi$ associated with $X$ has a completely monotone density, i.e., there is a density $\rho$ of the measure $\Pi$, which is infinitely differentiable on $(0,\infty)$ and such that $$(-1)^n\rho^{(n)}(x)\geq 0,$$ for arbitrary positive integer $n$ and $x> 0$. Then $\mathbf{W}^{(q)}$, the $q$-scale function of $Y$, can be expressed as
\begin{gather*}
    \mathbf{W}^{(q)}(x)
    =\phi^{\prime}(q)e^{\phi(q)x}-f(x),
    \quad x>0,
    \tag{2.5}
    \label{2.5}
\end{gather*}
where the function $f$ is a completely monotone function taking the form: $$f(x)=\int_{(0,\infty)}e^{-xt}\xi(dt),$$ in which $\xi$ is a bounded measure on $(0,\infty)$.
\end{lema}

\begin{lema}
\label{lem2.1}
Under the prerequisite given in Lemma \ref{lem2.0}, the $q$-scale function $W^{(q)}$ is infinitely differentiable on $(0,\infty)$, and its first derivative is strictly log-convex and thereby also convex on $(0,\infty)$. Furthermore, $W^{(q)\prime}$ decreases on $(0,a^*)$ and increases on $(a^*,\infty)$, where $a^*\geq 0$ is defined as the largest point at which the function $W^{(q)\prime}$ attains its global infimum and satisfies that $a^*<\infty$.
\end{lema}

\begin{coro}
\label{log-con}
Let the supposition of Lemma \ref{lem2.0} be true. Denote the $n$-th derivative of $W^{(q)}$ by $W^{(q),n}$. Consequently, $W^{(q),n}$ is strictly log-convex on $(0,\infty)$ if $n$ is a positive odd number.
\end{coro}

\begin{rem}
\label{FiN.rem.l}
Lemma \ref{lem2.0} is essentially the restatement of Lemma 4.17 in \cite{junca2019optimality}, Lemma \ref{lem2.1} stems from the discussion on $a^*$ and Lemma 3 in \cite{kyprianou2012optimal}, and Corollary \ref{log-con} comes from Corollary 4 in \cite{loeffen2009optimal}. Although Lemma \ref{lem2.1} concerns $W^{(q)}$, its conclusion applies to $\mathbf{W}^{(q)}$ as well because $X$ and $Y$ share the same jump measure $\Pi$. The smoothness and convexity of $W^{(q)}$ and $\mathbf{W}^{(q)}$ under the precondition of Lemma \ref{lem2.1} shall be used in the sequel without special justification. 
\end{rem}

The identity \eqref{2.7} in Lemma \ref{Lem2.2} below comes from \cite{kuznetsov2012theory}, and another part, \eqref{2.8}, whose rigorous proof has been given in Remark \ref{rem2.5}, relates to the $q$-scale functions' second derivative. As seen in Remark \ref{rem2.5}, the value given in \eqref{2.8} is a consequence of the restated part \eqref{2.7}, which can be checked by switching the limit and integral signs in \eqref{2.9} with $x\to 0^+$ in the equation. For the subsequent formulas concerning only the values of $W^{(q)\prime}$ and $W^{(q)\prime\prime}$ on $[0,\infty)$, $W^{(q)\prime}(0)$ and $W^{(q)\prime\prime}(0)$ therein are understood as the right-hand limit at the origin, respectively. Also, the right derivative notion $W^{(q)\prime}_+$ would occasionally have to be employed as $W^{(q)\prime}(0)$ is not well-defined, and by $W^{(q)\prime\prime}_+$ we denote the right derivative of $W^{(q)\prime}$.

\begin{lema}
\label{Lem2.2}
Let $X$ be a spectrally negative L\'evy process with the L\'evy triplet $(\gamma,\sigma,\Pi)$ s.t.
\begin{gather*}
    \sigma>0\quad \text{and} \quad \int_{(0,1)}z\Pi(dz)<\infty.
    \tag{2.6}
    \label{2.6}
\end{gather*}
Then the $q$-scale function for $X$, $W^{(q)}$, satisfies that 
\begin{gather*}
    W^{(q)}(0)=0,\quad
    W^{(q)\prime }(0+)
    :=
    \lim_{x\to 0^+}W^{(q)\prime }(x)
    =
    \frac{2}{\sigma^2},
    \tag{2.7}
    \label{2.7}
    \\
    W^{(q)\prime\prime}(0+)
    :=
    \lim_{x\to 0^+}W^{(q)\prime\prime }(x)
    =
    -\frac{4c}{\sigma^4}
    .
    \tag{2.8}
    \label{2.8}
\end{gather*}
\end{lema}

\begin{rem} 
\label{rem2.5}
The quantities for $W^{(q)}$ that appeared in \eqref{2.7} are quoted from Chapter 3 of \cite{kuznetsov2012theory}, which only requires the considered $X$ to have a Gaussian part, i.e., $\sigma>0$. By $W^{(q)}(0)=0$ from \eqref{2.7} and the fact that $W^{(q)}$ vanishes on $(-\infty,0)$, we have $W^{(q)}\equiv 0$ on $(-\infty,0]$, which is to be invoked constantly. 

The equality in \eqref{2.8} can be obtained by taking $x\to 0^+$ in the identity given by
\begin{gather*}
  (\mathcal{G}-q)W^{(q)}(x)
  =0,
  \quad 
  x>0.
  \tag{2.9}
  \label{2.9}
\end{gather*}
The identity \eqref{2.9} is available in Chapter 3 of \cite{kuznetsov2012theory}, meaning that the $q$-scale function $W^{(q)}$ is the eigen-function for the infinitesimal generator $\mathcal{G}$ with the eigen-value $q$. Let $a$ be a positive constant and suppose that \eqref{2.6} holds. To prove Lemma \ref{Lem2.2}, we first observe that 
\begin{gather*}
    \left| W^{(q)}(x-y)-W^{(q)}(x)+W^{(q)\prime}(x)y1_{\{0<y<1\}}\right|
    \\
    =\left| 
    W^{(q)}(x-y)-W^{(q)}(x)+W^{(q)\prime}(x)y
    \right|1_{\{0<y<1,y<x\}}
    \\
    +
    \left| 
    W^{(q)}(0)-W^{(q)}(x)+W^{(q)\prime}(x)y
    \right|1_{\{0<y<1,y\geq x\}}
    +
    \left|
    W^{(q)}(x)-W^{(q)}(x-y)
    \right|1_{\{y\geq 1\}}
    \\
    =
    \left|
    \int_0^y 
    \left[
    W^{(q)\prime}(x)
    -
    W^{(q)\prime}_{+}(x-s)
    \right]
   ds \right|1_{\{0<y<1\}}
   +
    \left|
    W^{(q)}(x)-W^{(q)}(x-y)
    \right|1_{\{y\geq 1\}}
    \\
    \leq 
    h(y):=
    2\sup_{x\in [0, a]}|W^{(q)\prime}(x)|
    y1_{\{0<y<1\}}
    +2
    \sup_{x\in [0, a]}|W^{(q)}(x)|
    1_{\{y\geq 1\}}
    ,
\end{gather*}
is satisfied for all $x\in [0,a]$ and $y\geq 0$ since $W^{(q)}\equiv 0$ on $(-\infty,0]$. Then, the assertion to be proved is that $\sup_{x\in [0, a]}|W^{(q)\prime}(x)|<\infty$ and $\sup_{x\in [0, a]}|W^{(q)}(x)|<\infty$. By the identity $W^{(q)\prime}(0+)=\frac{2}{\sigma^2}$ in \eqref{2.7}, we immediately deduce that $$|W^{(q)\prime}(x)|\leq \frac{2}{\sigma^2}+\varepsilon,$$ for $0\leq x<\delta(\varepsilon)$, where $\delta(\varepsilon)$ is sufficiently small and $\varepsilon>0$ is arbitrary selected. Also, we have $W^{(q)}\in C^2(0,\infty)$ if $\sigma>0$. (see, for example, the discussion in \cite{kyprianou2010convexity}) Since $W^{(q)\prime}$ is continuous on the compact interval $[\delta(\varepsilon),a]$, it holds that $\max_{x\in [\delta(\varepsilon),a]}\left|W^{(q)\prime}(x)\right|<\infty$. Therefore, the bound for the function $|W^{(q)\prime}|$ on $[0,a]$ can be chosen as $$\max\{\frac{2}{\sigma^2}+\varepsilon,\max_{x\in [\delta(\varepsilon),a]}\left|W^{(q)\prime}(x)\right|\}<\infty.$$ Similarly, we also have $\sup_{x\in [0, a]}|W^{(q)}(x)|<\infty$. In addition, by \eqref{2.6}, we obtain that
\begin{gather*}
   \int_{(0,\infty)}h(y)\Pi(dy)
    \\
    \leq 2
    \max\{
    \sup_{x\in [0, a]}
    \left|W^{(q)}(x)\right|
    ,
    \sup_{x\in [0, a]}
    \left|W^{(q)\prime}(x)\right|
    \}
    \left[
    \int_{(0,1)}z\Pi(dz)
    +
    \Pi[1,\infty)
    \right]
    <\infty,
\end{gather*}
where $\Pi[1,\infty)<\infty$ holds because of the definition of the measure $\Pi$. Thus, by applying the dominated convergence theorem to the integral component in $\lim_{x\to 0^+}(\mathcal{G}-q)W^{(q)}(x)$,
$$
\lim_{x\to 0^+}
\int_{(0,\infty)}\left[W^{(q)}(x-y)-W^{(q)}(x)+W^{(q)\prime}(x)y1_{\{0<y<1\}}\right]
\Pi(dy)=\int_{(0,1)}z\Pi(dz)W^{(q)\prime}(0+),
$$ is derived with the aid of the fact that $W^{(q)}$ vanishes on $(-\infty,0]$. Then letting $x\to 0+$ in \eqref{2.9} and employing \eqref{2.7}, one shall be able to deduce that
\begin{gather*}
   \lim_{x\to 0+}(\mathcal{G}-q)W^{(q)}(x)=
   cW^{(q)\prime}(0+)
   +\frac{\sigma^2}{2}W^{(q)\prime\prime}(0+)
   -qW^{(q)}(0)
   =
   c\frac{2}{\sigma^2}
   +\frac{\sigma^2}{2}W^{(q)\prime\prime}(0+)
   =0,
\end{gather*}
which gives the value of $W^{(q)\prime\prime}(0+)$. 
\end{rem}

\begin{prop}
\label{lem2.end}
Let the condition stated for $X$ given in \eqref{2.6} be strengthened to
\begin{gather*}
    \sigma>0\quad \text{and} \quad \Pi(0,\infty)<\infty.
    \tag{2.10}
    \label{2.10}
\end{gather*}
Assume that $W^{(q)}$ is three times continuously differentiable on $(0,\infty)$. It is such that 
\begin{gather*}
    \left(\mathcal{G}-q\right)W^{(q)\prime}_+(x)=0, \text{ for all  } x>0.
    \tag{2.11}
    \label{2.11}
\end{gather*}
\end{prop}
\begin{rem}
\label{rem.add}
The claim is to sketch the proof of Proposition \ref{lem2.end}. Suppose \eqref{2.10} is true, which means that the argument in Remark \ref{rem2.5} applies here as well, and the deduction hinges on the use of the resulting fact that $W^{(q)\prime}$ and $W^{(q)\prime\prime}$ are bounded on $[0,a]$ for any $a>0$. The bound for $W^{(q)\prime}$ is shown in Remark \ref{rem2.5}. The bound of $W^{(q)\prime\prime}$ can be achieved by using \eqref{2.8} and imitating the argument for $W^{(q)\prime}$ in Remark \ref{rem2.5}. Fix $a>0$ and $\varepsilon^{\prime}>0$ and suppose that $x\in [0,a]$, $\varepsilon\in [0,\varepsilon^{\prime}]$, and $y\geq 0$. Take the right derivative of $x$ in \eqref{2.9}. The action to switch the right derivative and integral signs therein shall be justified by mimicking Remark \ref{rem2.5} to employ the dominated convergence theorem, recalling $\Pi(0,\infty)<\infty$ and the bound of $W^{(q)\prime}$ and $W^{(q)\prime\prime}$, and making use of the following observation: the easy-to-check fact that
\begin{gather*}
    W^{(q)}(t)-W^{(q)}(s)=\int_s^tW^{(q)\prime}_+(z)dz,\text{  for  }-\infty<s\leq t<\infty,
    \tag{2.12}
    \label{2.12}
\end{gather*}
is satisfied as $W^{(q)}\equiv 0$ on $(-\infty,0]$, $W^{(q)}$ is continuous at the origin, and $W^{(q)}\in C^2(0,\infty)$,
\begin{gather*}
\left|\frac{W^{(q)}(x-y+\varepsilon)-W^{(q)}(x+\varepsilon)+W^{(q)\prime}(x+\varepsilon)y
}{\varepsilon}
-
\frac{
W^{(q)}(x-y)-W^{(q)}(x)+W^{(q)\prime}(x)y
}{\varepsilon}
\right|
\\
=
\left|\
\frac{\int_0^{\varepsilon}\left[
W_+^{(q)\prime}(x+n-y)
-
W^{(q)\prime}(x+n)
+W^{(q)\prime\prime}(x+n)
y
\right]dn }{\varepsilon}
\right|
\\
\leq 
\left[
2 \sup_{x\in [0,a+\varepsilon^{\prime}]}
\left|W^{(q)\prime}(x)\right|
+ \sup_{x\in [0,a+\varepsilon^{\prime}]}
\left|W^{(q)\prime\prime}(x) \right|
\right]
,\text{  for all  }0<y<1,
\end{gather*}
is true by \eqref{2.12} and the definition of an indicator function, and, again in view of \eqref{2.12}, 
\begin{gather*}
    \left|\frac{W^{(q)}(x-y+\varepsilon)-W^{(q)}(x+\varepsilon)
}{\varepsilon}
-
\frac{
W^{(q)}(x-y)-W^{(q)}(x)
}{\varepsilon}
\right|
\\
=
\left|\
\frac{\int_0^{\varepsilon}\left[
W_+^{(q)\prime}(x+n-y)
-
W^{(q)\prime}(x+n)
\right]dn }{\varepsilon}
\right|
\leq 2\sup_{x\in[0,a+\varepsilon^{\prime}]}\left|W^{(q)\prime}(x)\right|
,
\end{gather*}
holds for all $y\geq 1$. Hence, for all $x>0$, the integral part contained in $\left\{\left(\mathcal{G}-q\right)W^{(q)}\right\}^{\prime}_{+}(x)$ allows swapping the integral and right derivative signs, and the right derivative and derivative signs are interchangeable in the other part by the pre-specified condition that $W^{(q)}$ is three times continuously differentiable on $(0,\infty)$, which finalizes an outline of the proof.

\end{rem}

\begin{lema}
\label{lem2.229}
Let the condition for $X$ given in Proposition \ref{lem2.end} be true. Suppose that $W^{(q)}$ is three times continuously differentiable on $(0,\infty)$. Consequently, 
\begin{gather*}
    W^{(q)\prime\prime\prime}(0+)
    :=
    \lim_{x\to 0^+}W^{(q)\prime\prime\prime }(x)
    =
    \frac{4}{\sigma^4}
    \left(
    \Pi(0,\infty)+q+\frac{2c^2}{\sigma^2}
    \right)
    .
    \tag{2.13}
    \label{2.13}
\end{gather*}
\end{lema}

\begin{rem}
Assume that the prerequisite of Proposition \ref{lem2.end} is true. Through switching the limit and integral signs within the integral component in $\lim_{x\to 0^+}(\mathcal{G}-q)W^{(q)\prime}_+(x)$,
\begin{gather*}
\lim_{x\to 0^+}
\int_{(0,\infty)}\left[W^{(q)\prime}_+(x-y)-W^{(q)\prime}(x)+W^{(q)\prime\prime}(x)y1_{\{0<y<1\}}\right]
\Pi(dy)\\=-\Pi(0,\infty)W^{(q)\prime}(0+)+\int_{(0,1)}z\Pi(dz)W^{(q)\prime\prime}(0+),
\tag{2.14}
\label{2.14}
\end{gather*} is derived heuristically with the help of the fact that $W^{(q)}$ vanishes on $(-\infty,0)$. Substituting \eqref{2.14} back into $\lim_{x\to 0^+}(\mathcal{G}-q)W^{(q)\prime}_+(x)$ and invoking \eqref{2.7} and \eqref{2.8} would give \eqref{2.13}, whose rigorous proof is akin to the justification of \eqref{2.8} in Remark \ref{rem2.5} and relatively easier as the imposed condition here $\Pi(0,\infty)<\infty$ is stricter, and should be therefore skipped here. 
\end{rem}
\section{Main results}
\label{S.3}
We always discuss the process $X$ under the assumption indicated in Lemma \ref{lem2.0} and Lemma \ref{Lem2.2}, which means that the smoothness requirement of $W^{(q)}$ on $(0,\infty)$ in Proposition \ref{lem2.end} and Lemma \ref{lem2.229} is automatically satisfied since, by Lemma \ref{lem2.1}, $W^{(q)}$ is infinitely differentiable on $(0,\infty)$ if the measure $\Pi$ has a completely monotone density. These assumptions, coupled with the other necessary restriction, are stated formally as:
\begin{gather*}
\sigma>0, \quad
    \phi(q)<\frac{2\delta}{\sigma^2},
    \quad 
    \Pi(0,\infty)<\infty,
    \text{  and  } \Pi \text{ has a completely monotone density}.
    \\ \text{Furthermore,  }S<\frac{c}{\Pi(0,\infty)+q}.
    \tag{$\circledast$}
    \label{A}
\end{gather*}
Notably, Assumption \eqref{A} ensures that Lemma \ref{lem2.229} holds. From the inequality \eqref{3.2.17} in the sequel, it follows that $$c>\delta-\frac{1}{2}\sigma^2\phi(q)>0,$$ holds if $\phi(q)<\frac{2\delta}{\sigma^2}$ in Assumption \eqref{A} is satisfied. In what follows, we only deal with the case $S>0$, as mentioned in \eqref{1.1}, and would reiterate this assumption when it is necessarily used.
\subsection{Optimality condition for threshold strategies}
\begin{teor}
\label{teor1}
 Let Assumption \eqref{A} be satisfied. Then it is such that
\begin{itemize}
  \item If $S\in (0, 
 \frac{\frac{\delta}{\phi(q)}-\frac{\sigma^2}{2}}{c-
 \delta 
 +\phi(q)\frac{\sigma^2}{2}})$ holds, then the threshold strategy forms the optimal dividend strategy with the strictly positive threshold level at $b^*_S>0$. 
\end{itemize}
\end{teor}

\begin{rem}
The optimality of threshold dividend strategies in the case $S=0$ has been shown in \cite{kyprianou2012optimal}. 
\end{rem}

The selection of the point $b_S^*$ is accomplished by investigating a set of functions, $A_S$, $\theta_S$, $g_S$, and $r_S$, whose illustration is deferred to the subsection below.
\subsection{A criterion for selecting the threshold $b_S^*$}
Invoking Theorem 6, and Theorem 7 of \cite{kyprianou2010refracted}, the expected net present values defined in \eqref{1.1} with the threshold strategy executed can be expressed as: for all $b\geq 0$ and $x\in \mathbb{R}$, 
\begin{gather*}
V^S_b(x)=
    \mathbb{E}_x
    \left[ 
    \int_{[0,\tau_{\pi_b}]}
    e^{-qt}l_t^bdt
    \right]
    +S\mathbb{E}
    \left[ 
    e^{-q\tau_{\pi_b}}
    1_{\{U^b_{\tau_{\pi_b}}=0\}}
    \right]
    \\
    =S\frac{\sigma^2}{2}
    \left[
    W^{(q)\prime}(x)
    +
    \delta 
    \int_b^x \textbf{W}^{(q)}(x-y)W^{(q)\prime \prime}
    (y)dy
    \right]
    \\
    +A_S(b)\left[  
    W^{(q)}(x)
    +\delta \int_b^x \textbf{W}^{(q)}(x-y)
    W^{(q)\prime}(y)dy
    \right]
    -
    \delta 
    \overline
    {{\mathbf{W}}}^{(q)}(x-b)
    ,\quad S>0,
    \tag{3.2.1}
    \label{3.2.1}
\end{gather*}
in which the function $A_S$ is given by
\begin{equation}
    A_S(b)=
    \frac{
    \int_b^{\infty}
    e^{-\phi(q)y}\left(1-
    S\frac{\sigma^2}{2}
    W^{(q)\prime\prime}(y)\right)dy
    }  
    {\int_b^{\infty}e^{-\phi(q)y}W^{(q)\prime}(y)dy}
,
\quad 
S> 0,
\tag{3.2.2}
\label{3.2.2}
\end{equation}
for all $b\geq 0$. Note that, within \eqref{3.2.1}, $V^S_b(0)=\lim_{x\to 0^+}V^S_b(x)$. Define the function $\theta_S$ by
\begin{equation*}
    \theta_S(b)
=
    \frac{
1-S\frac{\sigma^2}{2}
W^{(q)\prime \prime }(b)}
{W^{(q)\prime}(b)}
,
\quad 
S> 0,
\tag{3.2.3}
\label{3.2.3}
\end{equation*}
for all $b\in (0,\infty)$.

Differentiating \eqref{3.2.2}, we have
\begin{gather*}
    A_S^{\prime}(b)
    =
    \frac{W^{(q)\prime}(b)}{\int_b^{\infty}e^{-\phi(q)(y-b)}W^{(q)\prime}(y)dy}
\left[  A_S(b)
    -
    \theta_S(b)
    \right],\text{  for all  }b\geq 0,
    \tag{3.2.4}
    \label{3.2.4}
\end{gather*}
where $A_S^{\prime}(0)=A_S^{\prime}(0+)$ and $\theta_S(0)=\theta_S(0+)$. Here we note that $\frac{W^{(q)\prime}(b)}{\int_b^{\infty}e^{-\phi(q)(y-b)}W^{(q)\prime}(y)dy}>0$ for all $b\geq 0$ since $W^{(q)}$ strictly increases on $[0,\infty)$ ($W^{(q)\prime}>0$ on $[0,\infty)$) and thereby that,
\begin{gather*}
\text{for all  }b\in [0,\infty), \quad
    A^{\prime}_S(b)>(<)~0\iff
    A_S(b)>(<)~\theta_S(b).
    \tag{3.2.5}
    \label{3.2.5}
\end{gather*}
Observe that the differential equation for $\theta_S$ holds:
\begin{gather*}
    \theta^{\prime}_S(b)
    =
    -\frac{W^{(q)\prime \prime }(b)}{W^{(q)\prime }(b)}
    \left[\theta_S(b)-g_S(b)\right],
    \quad 
     b\in (0,a^*)\cup(a^*,\infty),
\tag{3.2.6}
\label{3.2.6}
\end{gather*}
where the function $g_S$ is given by 
\begin{gather*}
g_S(b)=
-S\frac{\sigma^2}{2}\frac{W^{(q)\prime\prime \prime}(b)}{W^{(q)\prime\prime }(b)}
,
\quad 
S> 0,
\tag{3.2.7}
\label{3.2.7}
\end{gather*}
for $b\in (0,a^*)\cup(a^*,\infty)$. Taking derivatives on both sides of \eqref{3.2.7} deduces that
\begin{gather*}
    g_S^{\prime}(b)
    =-\frac{
    W^{(q)\prime\prime\prime}(b) 
    }{
    W^{(q)\prime\prime}(b)
    }
    \left[ 
    g_S(b)-r_S(b)
    \right],\quad 
    b\in (0,a^*)\cup (a^*,\infty),
    \tag{3.2.8}
    \label{3.2.8}
\end{gather*}
where the function $r_S$ takes the form:
\begin{gather*}
    r_S(b)=
-S\frac{\sigma^2}{2}\frac{W^{(q)\prime\prime \prime\prime}(b)}{W^{(q)\prime\prime\prime}(b)}
,
\quad 
S> 0,
\tag{3.2.9}
\label{3.2.9}
\end{gather*}
for $b\in (0,\infty)$.

\begin{lema}
\label{lem000}
Under Assumption \eqref{A}, given that $S>0$ and $a^*>0$, we have
\begin{gather*}
    g_S
    \text{  either firstly decreases on  } (0,o_S^*)
    \text{  and then increases  }(o_S^*,a^*)
    \\
    \text{  or increases on  }(0,a^*),
    \tag{3.2.10}
    \label{3.2.10}
\end{gather*}
where $o_S^*$ is a constant satisfying that $0<o_S^*< a^*$.
\end{lema}
\begin{proof}
Let $S>0$ and $a^*>0$. Thanks to \eqref{3.2.8}, it holds that
\begin{gather*}
\text{for  }S>0\text{  and  } b\in (0,a^*),\quad
g_S^{\prime}(b)>(<)~0\iff g_S(b)>(<)~r_S(b),
\tag{3.2.11}
\label{3.2.11}
\end{gather*}
based on the fact that $W^{(q)\prime}$ strictly decreases ($W^{(q)\prime\prime}<0$) on $(0,a^*)$ and that $W^{(q)\prime}$ is convex ($W^{(q)\prime\prime\prime}>0$) on $(0,\infty)$, all of which is lifted from Lemma \ref{lem2.1}. By Corollary \ref{log-con}, we must have that $r_S$ strictly decreases on $(0,\infty)$ because of the expression of $r_S$ in \eqref{3.2.9} and the log-convexity property of $W^{(q)\prime\prime\prime}$ on $(0,\infty)$. If $g_S(0+)\geq r_S(0+)$, it then holds that $g_S$ increases on $(0,a^*)$ based on \eqref{3.2.11} and the property that $r_S$ strictly decreases on $(0,\infty)$. From Lemma \ref{lem2.1}, $W^{(q)\prime}$ strictly decreases ($W^{(q)\prime\prime}<0$) on $(0,a^*)$ and strictly increases ($W^{(q)\prime\prime}>0$) on $(a^*,\infty)$, which entails that $W^{(q)\prime\prime}(a^*)=0$. If $g_S(0+)< r_S(0+)$, by the condition that $r_S$ strictly decreases on $(0,\infty)$ and the property that $$\lim_{b\to a^*-}g_S(b)=+\infty,$$ which is achieved on the basis of the definition of $g_S$ in \eqref{3.2.7} and the fact that $W^{(q)\prime\prime}<0$ on $(0,a^*)$, $W^{(q)\prime\prime}(a^*)=0$ and $W^{(q)\prime\prime\prime}>0$ on $(0,\infty)$, it follows from \eqref{3.2.11} that $g_S$ and $r_S$ would intersect once at the point $o_S^*$ satisfying that $0<o_S^*<a^*$. As a result, $g_S$ decreases on $(0,o_S^*)$ and increases on $(o_S^*,a^*)$ if $g_S(0+)< r_S(0+)$. Hence, \eqref{3.2.10} is proven for the case $S>0$ and $a^*>0$.

\end{proof}

\begin{rem}
Lemma \ref{lem000} is the key element of the survey in the sense that the counterpart of $g_S$ in \cite{junca2019optimality} and \cite{kyprianou2012optimal} features obviously fine properties, whereas in our case, it is not evident. This deduction for Lemma \ref{lem000} is accomplished mainly with the help of Corollary \ref{log-con}. Notably, Proposition \ref{lem2.end} and Lemma \ref{lem2.229} allow the expression in \eqref{2.13} to hold, which is advantageous in Lemma \ref{Lem3.4} hereinafter. Both Lemma \ref{lem000} and Lemma \ref{Lem3.4} are essentially motivated by the proof of Theorem 1 in \cite{loeffen2009optimal}.
\end{rem}

\begin{lema}
\label{Lem3.4}
Suppose that $S>0$. Let Assumption \eqref{A} be satisfied. Define the point $a_S^*$ in the following manner:
\begin{gather*}
a_S^*=
    \inf\{
    b\geq 0:
    \theta_S^{\prime}(b)\leq  0
    \},
\end{gather*}
with the convention that $\inf\emptyset=\infty$. Then it holds that
\begin{gather*}
        \theta_S\text{  increases on  }(0,a_S^*)
    \text{  and decreases on  }(a_S^*,\infty),
    \tag{3.2.12}
    \label{3.2.12}
\end{gather*}
where $a^*_S\leq a^*$.
\end{lema}
\begin{proof}
Differentiating \eqref{3.2.3} entails that, for all $b>0$, 
\begin{gather*}
    \theta_S^{\prime}(b)
    =
    \frac{
    -W^{(q)\prime\prime}(b)
    +S\frac{\sigma^2}{2}
    \left[ 
    (W^{(q)\prime\prime}(b))^2
    -W^{(q)\prime}(b)W^{(q)\prime\prime \prime}(b)
    \right]
    }
    {\left(W^{(q)\prime}(b)\right)^2}
.
    \tag{3.2.13}
    \label{3.2.13}
\end{gather*}
Moreover, thanks to \eqref{3.2.6}, it is such that 
\begin{gather*}
\text{for  }S>0\text{  and  } b\in (0,a^*),\quad
\theta_S^{\prime}(b)>(<)~0\iff \theta_S(b)>(<)~g_S(b),
\tag{3.2.14}
\label{3.2.14}
\end{gather*}
on the basis of the fact that $W^{(q)\prime}$ strictly decreases ($W^{(q)\prime\prime}<0$) on $(0,a^*)$ and that $W^{(q)}$ strictly increases ($W^{(q)\prime}>0$) on $(0,\infty)$.
\begin{itemize}
    \item 
Let $S>0$ and $a^*>0$. The expression between the square bracket in \eqref{3.2.13} is negative on $(0,\infty)$ due to the log-convexity property of $W^{(q)\prime}$ from Lemma \ref{lem2.1}. Since additionally, $W^{(q)\prime\prime}>0$ on $(a^*,\infty)$ holds in \eqref{3.2.13} based on Lemma \ref{lem2.1}, we obtain that $\theta_S$ strictly decreases on $(a^*,\infty)$ and therefore $a_S^*\leq a^*$. By the relation shown in the proof of Lemma \ref{lem000}: $$\lim_{b\to a^{*-}}g_S(b)=+\infty,$$ we have $a_S^*<a^*$ based upon \eqref{3.2.14}. Recalling \eqref{3.2.10}, we have that $g_S$ strictly increases on $(0,a^*)$ or $g_S$ firstly decreases on $(0,o_S^*)$ and then increases on $(o_S^*,a^*)$. 
\begin{itemize}
    \item In the first situation: $$ g_S\text{  strictly increases to 
 }\lim_{b\to a^{*-}}g_S(b)=+\infty\text{  on  } (0,a^*),\text{  }0<a^*,$$ by using the relation \eqref{3.2.14}, it holds that 
\begin{gather*}
    \text{either  }\theta_S\text{ intersects with  }g_S\text{  once on  }(0,a^*)\text{  }(\theta_S(0)>  g_S(0+)),\\ \text{or  } \theta_S\text{ strictly decreases on  }(0,\infty)\text{  }(\theta_S(0)\leq  g_S(0+)).
\end{gather*}
\item In the second case: 
\begin{gather*}
    g_S \text{  decreases on  } (0,o_S^*) \text{  and increases to  }\lim_{b\to a^{*-}}g_S(b)=+\infty \text{  on  } (o_S^*,a^*) ,\\\text{  }0<o_S^*<a^*,
\end{gather*} 
by invoking the relation 
\eqref{3.2.14}, either $\theta_S(0)\geq g_S(0+)$, which implies that $\theta_S$ increases on $(0,o_S^*)$ as $g_S$ decreases on $(0,o_S^*)$, and that
\begin{gather*}
    \theta_S\text{ and }g_S\text{ intersect once only on }(o_S^*,a^*),
\end{gather*}
or $\theta_S(0)< g_S(0+)$, which entails that
\begin{gather*}
    \text{either  }\theta_S\text{ intersects with  }g_S\text{ once on  }(0,o_S^*)\text{  and also once on  }
    (o_S^*,a^*)\\(\theta_S(o_S^*)> g_S(o_S^*)),\\ \text{or  } \theta_S\text{  strictly decreases on  }(0,\infty)\text{  }(\theta_S(o_S^*)\leq  g_S(o_S^*)).
\end{gather*} 
\end{itemize}
Recalling the restriction in Assumption \eqref{A} saying that $\Pi(0,\infty)<\infty$ and $S<\frac{c}{\Pi(0,\infty)+q}$ and the definition of $\theta_S$ and $g_S$ in \eqref{3.2.3} and \eqref{3.2.7}, correspondingly, we must have $$\theta_S(0)=\frac{\sigma^2}{2}+Sc>g_S(0+)=S\frac{\sigma^2}{2}\frac{\Pi(0,\infty)+q}{c}+Sc,$$ by using \eqref{2.7}, \eqref{2.8} and \eqref{2.13} to compute the value of $\theta_S(0)$ and $g_S(0+)$. Combining the aforementioned yields that $\theta_S$ and $g_S$ intersect with each other only once on $(0,\infty)$ at $a_S^*>0$. Thus, \eqref{3.2.12} is valid.
\item
    If $S>0$ and $a^*=0$. Implementing the preceding argument implies that $\theta_S$ strictly decreases on $(a^*,\infty)=(0,\infty)$. As a result, \eqref{3.2.12} is satisfied with $a^*_S=a^*=0$.
\end{itemize}

\end{proof}
\begin{prop}
\label{prop3}
Let $S>0$. Suppose that Assumption \eqref{A} holds. Define
\begin{gather*}
    b^{*}_S
    =\inf\{
    b\geq 0:
    A_S^{\prime}(b)\leq 0
    \}=
    \inf\{
    b\geq 0:
    A_S(b)\leq  \theta_S(b)
    \},
    \tag{3.2.15}
    \label{3.2.15}
\end{gather*}
with the convention that $\inf\emptyset=\infty$. Then we have
\begin{itemize}
    \item It holds that
    \begin{gather*}
    b^{*}_S>0\iff S <
 \frac{\frac{\delta}{\phi(q)}-\frac{\sigma^2}{2}}{c-
 \delta 
 +\phi(q)\frac{\sigma^2}{2}}.
    \end{gather*}.
    \item It is such that
    $$b^*_S\leq a_S^*\leq a^*<\infty.$$
\end{itemize}
\end{prop}
\begin{proof}
To begin with, we show the condition guaranteeing $b^*_S>0$.
\begin{itemize}
\item Implementing the identities \eqref{2.7} and \eqref{2.8} and the equations \eqref{A.1.4} and \eqref{A.1.5}, we have 
    \begin{gather*}
        A_S^{\prime}(0)>~0
        \iff
        A_S(0)=\frac{\delta}{\phi(q)}+S\left[\delta -\frac{\sigma^2}{2}\phi(q)\right]>~\frac{\sigma^2}{2}+Sc=\theta_S(0),
        \tag{3.2.16}
        \label{3.2.16}
    \end{gather*}
    for $S>0$. Using the property that $\psi_Y(\phi(q))=\psi_X(\phi(q))-\delta\phi(q)=q$, it can be inferred that 
\begin{gather*}
    c-\delta+\frac{1}{2}\sigma^2\phi(q)
    =\frac{
    q+\int_{(0,\infty)}
    \left(
    1-e^{-\phi(q)z}
    \right)\Pi(dz)
    }{\phi(q)}>0.
    \tag{3.2.17}
    \label{3.2.17}
\end{gather*}
As $b_S^*>0$ if and only if $A_S^{\prime}(0)>0$, we conclude the argument for $b_S^*>0$ by the derivation previously made in combination.
\end{itemize}
Next, we investigate the relationship between $b_S^*$ and $a_S^*$.
\begin{itemize}
    \item Simple transformation of the identity \eqref{3.2.4} provides the expression as follows:
    \begin{gather*}
        A_S^{\prime}(b)=
        \frac{
        \int_b^{\infty}
        e^{-\phi(q)y}
        W^{(q)\prime}(b)
         W^{(q)\prime}(y)
        \left[ 
        \theta_S(y)-
        \theta_S(b)
        \right]dy
        }{
        e^{\phi(q)b}
        (\int_b^{\infty}e^{-\phi(q)y}W^{(q)\prime}(y)dy)^2
        },
        \tag{3.2.18}
        \label{3.2.18}
    \end{gather*}
    for all $b\geq 0$. From \eqref{3.2.12} in Lemma \ref{Lem3.4}, we derive that $A_S^{\prime}\leq 0$ on $(a_S^*,\infty)$. Hence $b_S^*\leq a_S^*$. The fact that $a^*_S\leq a^*$ is based on Lemma \ref{Lem3.4}, and the relation $a^*<\infty$ comes from Lemma \ref{lem2.1}.
\end{itemize}
\end{proof}

\subsection{Verification}
We will verify Theorem \ref{teor1} step by step in the following text.

\begin{lema}
\label{Lem3.6}
Assume that $S>0$. Let $\pi$ be an admissible dividend strategy. Presume that the surplus process evolves in accordance with Assumption \eqref{A} and suppose that the function $V^S_{\pi}$ is twice continuously differentiable on $(0,\infty)$, and that the HJB equation as follows
\begin{gather*}
    (\mathcal{G}-q)V^S_{\pi}(x)
    +\sup_{0\leq r\leq \delta}
    r\left[1- V^{S\prime}_{\pi}(x)\right]\leq 0,
     \text{  for all  } x>0,
     \tag{3.3.1}
     \label{3.3.1}
\end{gather*}
holds. In addition, assume that $V^S_{\pi}\geq \Lambda_S$ on $(-\infty,0]$. Then $\pi$ is the optimal dividend strategy.
\end{lema}
\begin{proof}
We provide the sketch of the proof here. For $\pi$ defined in Lemma \ref{Lem3.6}, abbreviate $V^S_\pi(x)$ to $\varpi(x)$ for $x\in \mathbb{R}$. Following the proof of Lemma 5 in \cite{kyprianou2012optimal}, for the fixed dividend strategy $\pi_0 \in \mathcal{D}$, we shall have, for $x>0$, 
\begin{gather*}
    \varpi(x)
    =
    -\int_{0}^{t\wedge \tau_n}e^{-qs}
    \left[
    (\mathcal{G}-q)\varpi(U_{s-}^{\pi_0})
    +l^{\pi_0}_s(1- \varpi^{\prime}(U_{s-}^{\pi_0}))
    \right]ds
    +\int_{0}^{t\wedge \tau_n}
    e^{-qs}l^{\pi_0}_sds
                \\
    +e^{-qt\wedge \tau_n}\varpi(U_{t\wedge\tau_n }^{\pi_0})
    +M_t,
\end{gather*}
where $M=(M_t)_{t\geq 0}$ is a zero-mean martingale, $\{\tau_n\}_{n\geq 1}$ is some stopping time sequence that bounds the controlled surplus $U^{\pi_0}=(U^{\pi_0}_t)_{t\geq 0}$ on the closed interval $[\frac{1}{n},n]$. Let $\tau_{\pi_0}$ be the ruin time for the process $U^{\pi_0}$. In the formula above, implementing \eqref{3.3.1} and the fact that $\varpi(x)\geq 0$ for $x\neq 0$ ($\varpi\geq S>0$ on $(0,\infty)$ and $\varpi\geq 0$ on $(-\infty,0)$), and that $\varpi(0)\geq S$, taking the expectation under $\mathbb{E}_x$, and allowing $t$ and $n$ to go to infinity (dominated convergence theorem), we can reveal that $\varpi(x)=V^S_{\pi}(x)$ satisfies
\begin{gather*}
    \varpi(x)
    \geq
    V^S_{\pi_0}(x)
    =
    \mathbb{E}_x
    \left[  
    \int_{0}^{\tau_{\pi_0}}
    e^{-qs}l^{{\pi_0}}_sds
    +
    S
    e^{-q \tau_{\pi_0}}
     1_{\{
    U^{\pi_0}_{\tau_{\pi_0}}=0
    \}}
    \right],
\end{gather*}
as $\tau_n \nearrow \tau_{\pi_0}\text{ a.s.}$ under $\mathbb{P}_x$. Since the choice of $\pi_0\in \mathcal{D}$ is arbitrary, then $\varpi=V^S_{\pi}\geq V^S$ on $(0,\infty)$. To extent $\varpi\geq V^S$ to the domain $[0,\infty)$, again, see the argument in the proof of Lemma 5 in \cite{kyprianou2012optimal}. In this way, the optimality of $\pi$ follows. 
\end{proof}

Prior to showing the optimality for threshold strategies, the corresponding value function $V^S_{b_S^*}$ shall be proven to be sufficiently smooth on $(0,\infty)$.
\begin{lema}
\label{LEM32AA}
Let Assumption \eqref{A} be satisfied. Then the value function $V^S_{b^*_S}$ is twice continuously differentiable on $(0,\infty)$. 
\end{lema}
\begin{proof}
Suppose that $b^*_S=0$ hold first. Differentiating two sides of \eqref{2.3} once and twice, as well as using integral by parts each time, yields that 
\begin{gather*}
    W^{(q)}(x)+\delta \int_0^x\mathbf{W}^{(q)}(x-y)W^{(q)\prime}(y)dy=\mathbf{W}^{(q)}(x),
    \tag{3.3.2}
    \label{3.3.2}
    \\
    W^{(q)\prime}(x)+\delta \int_0^x\mathbf{W}^{(q)}(x-y)W^{(q)\prime\prime}(y)dy
    =
    \mathbf{W}^{(q)\prime }(x)
    -\delta \frac{2}{\sigma^2}
    \mathbf{W}^{(q)}(x),
    \tag{3.3.3}
    \label{3.3.3}
\end{gather*}
where we used the identities $W^{(q)}(0+)=0$ and $W^{(q)\prime}(0+)=\frac{2}{\sigma^2}$ from \eqref{2.7} in Lemma \ref{Lem2.2}. Plugging \eqref{3.3.2} and \eqref{3.3.3} into \eqref{3.2.1} deduces that
\begin{gather*}
V^S_{b_S^*}(x)=
   V^S_0(x)=
   S\frac{\sigma^2}{2}
   \mathbf{W}^{(q)\prime}(x)+
   \left(
   A_S(0)-S\delta 
   \right)\mathbf{W}^{(q)}(x)
   -\delta
   \overline{\mathbf{W}}^{(q)}(x).
   \tag{3.3.4}
   \label{3.3.4}
\end{gather*}
Obviously, $V^S_0$ is twice continuously differentiable on $(0,\infty)$ based on the smoothness property of $\mathbf{W}^{(q)}$ from Remark \ref{FiN.rem.l}.

Assume that $b^*_S>0$. Differentiating \eqref{3.2.1} once on $(0,b^*_S)\cup (b^*_S,\infty)$ and implementing integral by parts shows that 
\begin{gather*}
    V^{S\prime}_{b^*_{S}}(x)
    =S\frac{\sigma^2}{2}
    \left(
    W^{(q)\prime \prime }(x)
    +\delta \int_{b^*_{S}}^x
    \textbf{W}^{(q)\prime}(x-y)W^{(q)\prime \prime}(y)dy
    \right)
    -\delta \textbf{W}^{(q)}(x-b^*_{S})
    \\
    +
    A_S(b^*_{S})
    \left(
    W^{(q)\prime }(x)
    +\delta \int_{b^*_{S}}^x
    \textbf{W}^{(q)\prime}(x-y)W^{(q) \prime}(y)dy
    \right)
\\
    =S\frac{\sigma^2}{2}
    \left(
    W^{(q)\prime \prime }(x)
    +\delta \int_{b^*_S}^x
    \textbf{W}^{(q)}(x-y)W^{(q)\prime \prime \prime}(y)dy
    \right)
    \\
    +
    A_S(b^*_S)
    \left(
    W^{(q)\prime }(x)
    +\delta \int_{b^*_S}^x
    \textbf{W}^{(q)}(x-y)W^{(q)\prime \prime}(y)dy
    \right)
        \\
    +\delta \mathbf{W}^{(q)}(x-b^*_S)
    W^{(q)\prime}(b^*_S)
    \left[A_S(b^*_S)-\theta_S(b^*_S)\right],
    \tag{3.3.5}
    \label{3.3.5}
\end{gather*}
and differentiating \eqref{3.3.5} on $(0,b^*_S)\cup (b^*_S,\infty)$ would yield that 
\begin{gather*}
     V^{S\prime\prime}_{b^*_{S}}(x)
     =S\frac{\sigma^2}{2}
    \left(
    W^{(q)\prime \prime \prime}(x)
    +\delta \int_{b^*_S}^x
    \textbf{W}^{(q)\prime}(x-y)W^{(q)\prime \prime \prime}(y)dy
    \right)
    \\
    +
    A_S(b^*_S)
    \left(
    W^{(q)\prime \prime}(x)
    +\delta \int_{b^*_S}^x
    \textbf{W}^{(q)\prime}(x-y)W^{(q)\prime \prime}(y)dy
    \right)
        \\
    +\delta \mathbf{W}^{(q)\prime}(x-b^*_S)
    W^{(q)\prime}(b^*_S)
    \left[A_S(b^*_S)-\theta_S(b^*_S)\right].
        \tag{3.3.6}
    \label{3.3.6}
\end{gather*}
Since $A_S(b_S^*)=\theta_S(b_S^*)$ given that $b_S^*>0$, the continuity for $V^{S\prime }_{b^*_{S}}$ and $V^{S\prime \prime}_{b^*_{S}}$ at $b^*_{S}$ is also satisfied:
\begin{gather*}
    V^{S\prime }_{b^*_{S}}(b^*_{S}-)=V^{S\prime }_{b^*_{S}}(b^*_{S}+),
    \quad 
    V^{S\prime \prime}_{b^*_{S}}(b^*_{S}-)=V^{S\prime\prime}_{b^*_{S}}(b^*_{S}+),
\end{gather*}
due to the expression in \eqref{3.3.5} and \eqref{3.3.6}. Combining the deduction made with the fact that $q$-scale functions are infinitely differentiable from Remark \ref{FiN.rem.l}, we obtain that $V^S_{b_S^*}$ is twice continuously differentiable on $(0,\infty)$.

\end{proof}

Lemma \ref{lem3.7} below is useful as it gives a more tractable equation for the value function that is equivalent to the HJB equation in \eqref{3.3.1}. The proof of Lemma \ref{lem3.7} is omitted here, a standard argument for this could be found in Lemma 7 of \cite{kyprianou2012optimal}.
\begin{lema}
\label{lem3.7}
The value function $V^S_{b^*_S}$ satisfies \eqref{3.3.1} if and only if 
\begin{gather*}
    V^{S\prime}_{b^*_S}(x)\geq 1,\quad \text{if}\quad 0<x\leq b^*_S,
    \\
    V^{S\prime}_{b^*_S}(x)\leq 1,\quad \text{if}\quad x> b^*_S.
    \tag{3.3.7}
     \label{3.3.7}
\end{gather*}
\end{lema}

Lemma \ref{lem3.8} would play a crucial role in verifying \eqref{3.3.7}. To prove Lemma \ref{lem3.8} as follows, it is essential to recall that, by Assumption\eqref{A}, $\Pi$ has a completely monotone density. 

\begin{lema}
\label{lem3.8}
Suppose that $S>0$. Let Assumption \eqref{A} hold. If $b_S^*>0$, then $V^{S\prime}_{b^*_S}$ admits the following expression:
 \begin{gather*}
     V^{S\prime}_{b^*_S}(x)
     =\int_{(0,\infty)}e^{-xz}p_S(z)\xi(dz),
      \text{  for all  } x>b^*_S,
     \tag{3.3.8}
     \label{3.3.8}
 \end{gather*}
 where the function $p_S$ is given by
 \begin{gather*}
     p_S(z)=
     -S\frac{\sigma^2}{2}
     \left(
     z^2
     +\delta \frac{2}{\sigma^2}z
     +\delta z 
     \int_0^{b^*_S}
     e^{zy}W^{(q)\prime\prime}(y)dy
     \right)+\delta e^{b_S^*z}
              \\
     +A_S(b^*_S)
     \left(
     z-
     \delta z
     \int_0^{b^*_S}
     e^{zy}W^{(q)\prime}(y)dy
     \right)
     ,\quad S>0.
     \tag{3.3.9}
     \label{3.3.9}
 \end{gather*}
Also, given that $S>0$ and $b_S^*>0$, $p_S$ is a concave function and
 $V^S_{b^*_S}$ is such that 
 \begin{gather*}
     V_{b^*_S}^{S\prime\prime}(x)\leq 
     e^{(x-b^*_S)\beta}
     V_{b^*_S}^{S\prime\prime}(b^*_S+).
     \tag{3.3.10}
     \label{3.3.10}
 \end{gather*}
 for some point $\beta\in (0,\infty]$ and $x>b^*_S$.
\end{lema}
\begin{proof}
The derivation of the exact expression of $V^{S\prime}_{b_S^*}$ is in Appendix \ref{A.3}.

Let $S>0$. Differentiating \eqref{3.3.9} gives the following identity:
\begin{gather*}
    p_S^{\prime\prime}(z)
    =-S{\sigma^2}
    -\frac{S{\sigma^2}\delta}{2} 
    \int_0^{b^*_S}e^{zy}(2y+zy^2)W^{(q)\prime\prime}(y)dy+
    \delta e^{b^*_Sz}(b^*_S)^2
        \\
    -A_S(b^*_S)\delta \int_0^{b^*_S}e^{zy}(2y+zy^2)W^{(q)\prime}(y)dy.
\end{gather*}
In view of the definition of $b_S^*$ and the property that $A_S^{\prime}(z)\geq 0$ is equivalent to $A_S(z)\geq \theta_S(z)$ from \eqref{3.2.5}, we have $A_S(z)\geq \theta_S(z)$ for $b\in [0,b_S^*]$. Then we obtain 
\begin{gather*}
    p_S^{\prime\prime}(z)\leq 
    -S{\sigma^2}
    -\frac{S{\sigma^2}\delta}{2} 
    \int_0^{b^*_S}e^{zy}(2y+zy^2)W^{(q)\prime\prime}(y)dy
    +
    \delta e^{b^*_Sz}(b^*_S)^2
            \\
    -\delta \int_0^{b^*_S}e^{zy}(2y+zy^2)\left[1-S\frac{\sigma^2}{2}W^{(q)\prime\prime}(y)\right]dy=-S\sigma^2<0,
\end{gather*}
for $z\in [0,b_S^*]$. Obviously, we are able to conclude that $p$ is a concave function when $S>0$. Note that $p_S(0)=\delta>0$ because of \eqref{3.3.9}. Hence there exists a point $\beta\in (0,\infty]$ such that the function $p$ is positive on $(0,\beta)$ and negative $(\beta,\infty)$. Then $e^{-(x-b^*_S)z}p(z)\geq e^{-(x-b^*_S)\beta}p(z)$ holds for all $z>0$. In this way, we can see that
\begin{gather*}
    V_{b^*_S}^{S\prime\prime}(x)=
    -\int_{(0,\infty)}
    e^{-(x-b^*_S)z}e^{-b^*_Sz}zp_S(z)\xi(dz)\leq 
    -
     e^{(x-b^*_S)\beta}
    \int_{(0,\infty)}
   e^{-b^*_Sz}zp_S(z)\xi(dz)
    \\
    =
     e^{(x-b^*_S)\beta}
     V_{b^*_S}^{S\prime\prime}(b^*_S+),
     \text{  for all  }
     x>b_S^*.
     \end{gather*}

Note that a similar argument can be found in Lemma 4.20 of \cite{junca2019optimality} and Lemma 8 of \cite{kyprianou2012optimal}.
\end{proof}

Finally, the following result would make use of the foregoing and help arrive at the optimality of threshold strategies. 

\begin{coro}
Suppose that Assumption \eqref{A} holds. Let $S\in (0, 
 \frac{\frac{\delta}{\phi(q)}-\frac{\sigma^2}{2}}{c-
 \delta 
 +\phi(q)\frac{\sigma^2}{2}})$. Then the value function $V^S_{b^*_S}$ satisfies \eqref{3.3.7} and also \eqref{3.3.1}. Consequently, Theorem \ref{teor1} is valid.
\end{coro}
\begin{proof}
\begin{itemize}
    \item The first assertion is to show that \eqref{3.3.7} holds. As a consequence of $b_S^*\leq a_S^*$ in Proposition \ref{prop3} and \eqref{3.2.12}, it is such that 
\begin{gather*}
    \theta_S^{\prime}\geq 0~(\theta_S\text{  is non-decreasing}),\text{  on  }[0,b_S^*]\text{  with  }b_S^*\leq a_S^*.
\tag{3.3.11}
\label{3.3.11}
\end{gather*}
$S\in (0,\frac{\frac{\delta}{\phi(q)}-\frac{\sigma^2}{2}}{c-\delta+\phi(q)\frac{\sigma^2}{2}})$ is satisfied. Observe that $b_S^*>0$ in this case due to Proposition \ref{prop3}. If $x\leq b^*_S$, by using \eqref{3.3.5} and implementing the identity $A_S(b_S^*)=\theta_S(b_S^*)$ because of the definition of $b_S^*$ in \eqref{3.2.15} and the relation \eqref{3.2.5}, we could obtain the expression as follows: 
\begin{gather*}
V^{S\prime}_{b^*_S}(x)
=S\frac{\sigma^2}{2}W^{(q)\prime\prime}(x)
+
A_S (b^*_S)
W^{(q)\prime }(x)
= 
S\frac{\sigma^2}{2}W^{(q)\prime\prime}(x)
+
\theta_S (b^*_S)
W^{(q)\prime }(x)
\\
\geq 
S\frac{\sigma^2}{2}W^{(q)\prime\prime}(x)
+
\theta_S (x)
W^{(q)\prime }(x)
= 1,
\tag{3.3.12}
\label{3.3.12}
\end{gather*}
where the inequality holds due to \eqref{3.3.11}, and the last identity is thanks to the definition of $\theta_S$ in \eqref{3.2.3}. Also, it is easy-to-check that the equality in \eqref{3.3.12} can be achieved when $x=b^*_S$, meaning that $V^{S\prime}_{b^*_S}(b^*_S)=1$. Differentiating $\theta_S$ given in \eqref{3.2.3} presents that 
\begin{gather*}
    -\theta_S^{\prime}(b)W^{(q)\prime}(b)
    =W^{(q)\prime\prime}(b)\theta_S(b)
    +S\frac{\sigma^2}{2}W^{(q)\prime\prime\prime}(b), \text{  for all }b>0.
    \tag{3.3.13}
    \label{3.3.13}
\end{gather*}
Making $x\to b_S^*+$ in \eqref{3.3.6} shows that 
\begin{gather*}
    V^{S\prime\prime}_{b^*_S}(b^*_S+)
    =
    W^{(q)\prime\prime}(b^*_S+)
    A_S(b^*_S+)+S\frac{\sigma^2}{2}W^{(q)\prime\prime\prime}(b^*_S+)
    \\
    =W^{(q)\prime\prime}(b^*_S+)
    \theta_S(b^*_S+)+S\frac{\sigma^2}{2}W^{(q)\prime\prime\prime}(b^*_S+)
    ,
    \tag{3.3.14}
    \label{3.3.14}
\end{gather*}
where the last equality holds owing to the identity that $A_S(b_S^*)=\theta_S(b_S^*)$. By the expression \eqref{3.3.14}, using \eqref{3.3.11}, \eqref{3.3.13}, and the property that $W^{(q)}$ strictly increases ($W^{(q)\prime}>0$) on $(0,\infty)$ implies that 
\begin{gather*}
    V^{S\prime\prime}_{b^*_S}(b^*_S+)
    =
    W^{(q)\prime\prime}(b^*_S+)
    \theta_S(b^*_S+)
    +S\frac{\sigma^2}{2}W^{(q)\prime\prime\prime}(b^*_S+)
    =
    -\theta_S^{\prime}(b^*_S+)W^{(q)\prime}(b^*_S+)
    \leq 0.
    \tag{3.3.15}
    \label{3.3.15}
\end{gather*}
\eqref{3.3.15}, coupled with \eqref{3.3.10} in Lemma \ref{lem3.8}, deduces that $V^{S\prime\prime}_{b^*_S}(x)\leq 0$ for $x> b^*_S$ ($V^{S\prime}_{b^*_S}$ is non-increasing on $(b^*_S,\infty)$). As a result, $V^{S\prime}_{b^*_S}\leq 1$ on $(b^*_S,\infty)$ as $V^{S\prime}_{b^*_S}(b^*_S)=1$ by \eqref{3.3.12}.
\item 
In summary, we have demonstrated that $V^S_{b_S^*}$ indeed satisfies \eqref{3.3.7}, which, by Lemma \ref{lem3.7}, is equivalent to \eqref{3.3.1}. From \eqref{3.2.1}, it can be inferred that $V^S_{b^*_S}$ satisfies that $V^S_{b^*_S}(0)=S$ and $V^S_{b^*_S}(x)=0$ for $x<0$ by using the identity $W^{(q)\prime}(0+)=\frac{2}{\sigma^2}$ from \eqref{2.7} and the fact that $W^{(q)}$ vanishes on $(-\infty,0]$ and that $\mathbf{W}^{(q)}$ vanishes on $(-\infty,0)$. The foregoing, compounded by the smoothness property in Lemma \ref{LEM32AA}, finalizes the verification of the sufficient condition concluding the optimality in Lemma \ref{Lem3.6}. 
\end{itemize}

\end{proof}
\section{A numerical example}
\label{S.4}
\subsection*{Case Study}
Given the drift parameter $\mu\in \mathbb{R}$, the volatility parameter $\sigma>0$, a standard Brownian motion $B=(B_t)_{t\geq 0}$, \text{i.i.d.} exponential random variables $\{J_i\}_{i\geq 1}$ with parameter $p$, and a Poisson process $N=(N_t)_{t\geq 0}$ with arrival rate $\lambda>0$, $X$ is defined as 
 \begin{gather*}
     X_t=\mu t+\sigma B_t-\sum_{i=1}^{N_t}J_i,\quad t\geq 0.
 \end{gather*}
Here, $X$ is a special case of spectrally negative L\'evy processes. The corresponding jumping measure $\Pi$ is with a completely monotone density as 
 $\Pi(dx)=\rho(x)dx=\lambda pe^{-px}dx,\text{  for  }x>0,$ in which $(-1)^n\rho^{(n)}(x)=\lambda p^{n+1}e^{-px}> 0$ for $x>0$ and all positive integers $n$. Furthermore, it is easy to check that $\mu=c$, $\sigma$ corresponds to the parameter of the Gaussian part, and $\Pi(0,\infty)=\lambda$. 
 
We choose the aforesaid parameters in the following manner:
\begin{gather*}
    (\mu,\sigma,\lambda,p,S,q,\delta)=
    (2,1,1,0.5,0.05 ,4,1.8)
    .
    \tag{4.1}
    \label{4.1}
\end{gather*}
It can be verified the choice in \eqref{4.1} allows the hypothesis in Theorem \ref{teor1} to hold. The associated $q$-scale functions $W^{(q)}$ and $\mathbf{W}^{(q)}$ are expressed as 
\begin{gather*}
    W^{(q)}(x)\approx
    -0.264203 e^{-5.76694 x}
    -0.015547 e^{-0.4129 x}
    +0.27975 e^{1.67984 x},
    \\
    \mathbf{W}^{(q)}(x)\approx
    -0.304813 e^{-3.42214 x}
    -0.0199203 e^{-0.4 x}
    +0.324733 e^{2.82214 x}.
\end{gather*}
Based on the software Matlab, it is suggested that 
\begin{gather*}
    b^*_S\approx 0.01993, \quad 
    A_S(b_S^*)=\theta_S(b_S^*)\approx0.6341.
\end{gather*}
In the following graph (Figure 1: Case 1), $A_S$ is the blue line, and $\theta_S$ is the red line. As can be seen, they intersect with each other at the maximum point of $A_S$.
\begin{figure}[h]
    \centering
    \includegraphics[width=0.95\textwidth]{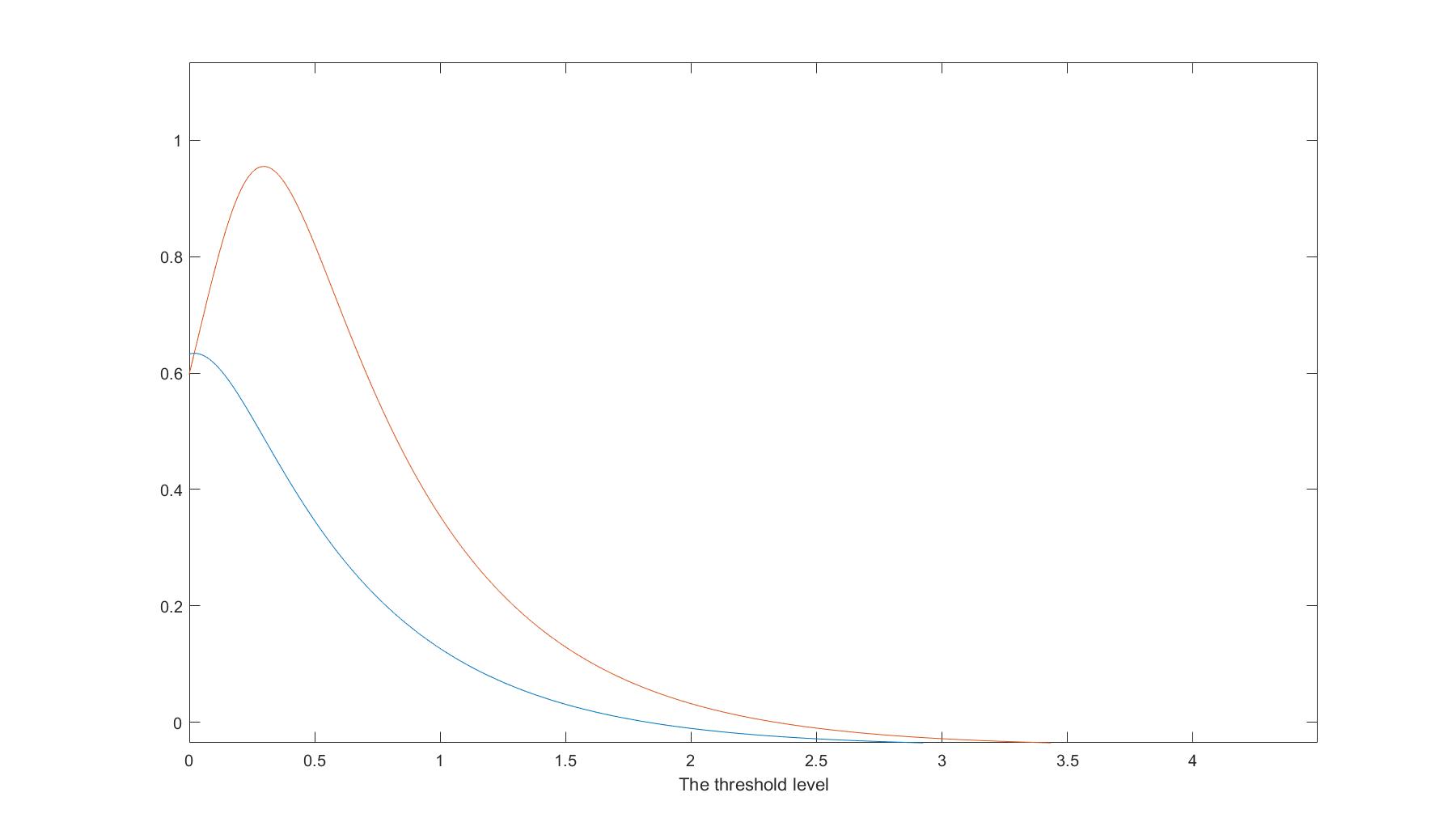}
    \caption{Case 1}
    \label{N1}
\end{figure}

\section{Conclusion}
\label{S.5}
In this research work, we investigate the optimal dividends problem when a positive terminal value at creeping ruin is incurred in the expected net present values considered. The optimality of the threshold strategy is shown in that case. The difficulty of investigating this problem lies in identifying the appropriate sufficient condition and proving that could ensure the usefulness of the rule to select the optimal threshold level, see \cite{junca2019optimality}
and \cite{kyprianou2012optimal}. Also, we offer one numerical example in the positive terminal value case where the surplus process evolves as a Brownian motion with drift and negative exponential jumps and give a visual expression along with it, which shows how the threshold level is chosen.

\section*{Appendix}\label{APPe}
\renewcommand\thesubsection{A.\arabic{subsection}}
\subsection{Proof of Lemma \ref{lem3.8}}
\label{A.3}
Let $S>0$. After differentiating two sides of \eqref{3.3.2} and \eqref{3.3.3}, we obtain that 
\begin{gather*}
W^{(q)\prime}(x)+\delta\int_0^x\mathbf{W}^{(q)\prime}(x-y)W^{(q)\prime}(y)dy=\mathbf{W}^{(q)\prime}(x),
\tag{A.1.1}
\label{A.1.1}
\\
    W^{(q)\prime \prime}(x)
    +\delta\int_0^x
    \mathbf{W}^{(q)\prime}(x-y)W^{(q)\prime\prime}(y)dy
    =
    \mathbf{W}^{(q)\prime\prime}(x)
    -\delta \frac{2}{\sigma^2}
    \mathbf{W}^{(q)\prime}(x).
    \tag{A.1.2}
    \label{A.1.2}
\end{gather*}
Differentiating \eqref{3.2.1}, inserting \eqref{A.1.1} and $\eqref{A.1.2}$ into the obtained expression, and then using the decomposition of $\mathbf{W}^{(q)}$ in \eqref{2.5} entails that 
\begin{gather*}
    V^{S\prime}_{b^*_S}(x)
    =S\frac{\sigma^2}{2}
    \left(
    W^{(q)\prime \prime }(x)
    +\delta \int_{b^*_S}^x
    \textbf{W}^{(q)\prime}(x-y)W^{(q)\prime  \prime}(y)dy
    \right)
    \\
    +
    A_S(b^*_S)
    \left(
    W^{(q)\prime }(x)
    +\delta \int_{b^*_S}^x
    \textbf{W}^{(q)\prime}(x-y)W^{(q) \prime}(y)dy
    \right)
        \\
    -\delta \mathbf{W}^{(q)}(x-b_S^*)
\\
    =S\frac{\sigma^2}{2}
    \left(
    \mathbf{W}^{(q)\prime\prime}(x)
    -\delta \frac{2}{\sigma^2}
    \mathbf{W}^{(q)\prime}(x)
    -\delta 
    \int_0^{b_S^*}
    \mathbf{W}^{(q)\prime}(x-y)W^{(q)\prime\prime}(y)dy
    \right)
    \\
    +
    A_S(b^*_S)
    \left(
    \mathbf{W}^{(q)\prime }(x)
    -\delta \int_0^{b^*_S}
    \textbf{W}^{(q)\prime}(x-y)W^{(q) \prime}(y)dy
    \right)
        \\
    -\delta \mathbf{W}^{(q)}(x-b_S^*)
\\
    =
    S\frac{\sigma^2}{2}
    \Big( 
    \left[ 
    \phi^{\prime}(q)(\phi(q))^2e^{\phi(q)x}
    -f^{\prime\prime}(x)
    \right]
    -
    \delta
    \frac{2}{\sigma^2}
    \left[\phi^{\prime}(q)\phi(q)e^{\phi(q)x}
    -f^{\prime}(x)
    \right]
    \\
    -\delta \int_0^{b^*_S}
    \left[\phi^{\prime}(q)\phi(q)e^{\phi(q)(x-y)}
    -f^{\prime}(x-y)
    \right]
    W^{(q)\prime\prime}(y)dy
    \Big)
    \\
    +
    A_S(b^*_S)
    \Big(
    \left[\phi^{\prime}(q)\phi(q)e^{\phi(q)x}
    -f^{\prime}(x)
    \right]
    -\delta \int_0^{b^*_S}
    \left[\phi^{\prime}(q)\phi(q)e^{\phi(q)(x-y)}
    -f^{\prime}(x-y)
    \right]
    W^{(q)\prime }(y)dy
    \Big)
    \\
    -\delta \Big(
    \phi^{\prime}(q)e^{\phi(q)(x-b_S^*)}
    -f(x-b_S^*)
    \Big)
    .
    \tag{A.1.3}
    \label{A.1.3}
\end{gather*}
Notice that 
\begin{gather*}
    \int_0^{\infty}e^{-\phi(q)y}W^{(q)}(y)dy=\frac{1}{\delta\phi(q)},
    \tag{A.1.4}
    \label{A.1.4}
\end{gather*}
holds due to the definition of $W^{(q)}$ in \eqref{2.1}. Recalling \eqref{2.4} and employing integral by parts once and twice with respect to \eqref{A.1.4} implies that
\begin{gather*}
    \int_0^{\infty}e^{-\phi(q)y}W^{(q)\prime}(y)dy=\frac{1}{\delta},
    \quad 
    \int_0^{\infty}e^{-\phi(q)}W^{(q)\prime\prime}(y)dy
    =\frac{\phi(q)}{\delta}-\frac{2}{\sigma^2}.
    \tag{A.1.5}
    \label{A.1.5}
\end{gather*}

Making use of \eqref{A.1.5} and carefully rearranging the terms in \eqref{A.1.3} gives that  
\begin{gather*}
    V^{S\prime}_{b_S^*}(x)
    =S\frac{\sigma^2}{2}
    \left(
    -f^{\prime\prime}(x)
    +\delta\frac{2}{\sigma^2}
    f^{\prime}(x)
    +\delta\int_0^{b_S^*}
    f^{\prime}(x-y)W^{(q)\prime\prime}(y)dy
    \right )
    \\
    +A_S(b_S^*)
    \left( 
    -f^{\prime}(x)
    +\delta 
    \int_0^{b_S^*}
    f^{\prime}(x-y)W^{(q)\prime}(y)dy
    \right) +\delta f(x-b_S^*)
    \\+
    \delta \phi^{\prime}(q)\phi(q)e^{\phi(q)x}
    \left[ 
    A_S(b_S^*)\int_{b_S^*}^{\infty}
    e^{-\phi(q)y}W^{(q)\prime}(y)dy 
    -\int_{b_S^*}^{\infty}
    e^{-\phi(q)y}
    \left(
    1-S\frac{\sigma^2}{2}
    W^{(q)\prime\prime}(y)
    \right)dy
    \right] 
    \\=S\frac{\sigma^2}{2}
    \left(
    -f^{\prime\prime}(x)
    +\delta\frac{2}{\sigma^2}
    f^{\prime}(x)
    +\delta\int_0^{b_S^*}
    f^{\prime}(x-y)W^{(q)\prime\prime}(y)dy
    \right )
    \\
    +A_S(b_S^*)
    \left( 
    -f^{\prime}(x)
    +\delta 
    \int_0^{b_S^*}
    f^{\prime}(x-y)W^{(q)\prime}(y)dy
    \right) +\delta f(x-b_S^*)
    =\int_{(0,\infty)}
    e^{-xz}p_S(z)\xi(dz),
\end{gather*}
where $f$ is the complete monotone function given in \eqref{2.5} in Lemma \ref{lem2.0}, and the last equality is valid owing to the definition of $A_S$.

\section*{Acknowledgement}
I am truly grateful to the two anonymous referees for giving their precious guiding comments for the work.

\end{document}